\documentclass[a4paper, 11pt, reqno]{amsart}
\usepackage{etex}
\usepackage{xypic}
\usepackage{amssymb}
\usepackage{amsthm}
\usepackage{bm}
\usepackage{latexsym}
\usepackage{amsmath}
\usepackage{eufrak}
\usepackage{mathrsfs}
\usepackage{amscd}
\usepackage[all]{xy}
\usepackage[usenames]{color}
\usepackage{graphicx}
\usepackage{color}
\usepackage{caption}
\usepackage{amscd}
\theoremstyle{plain}
\usepackage[usenames,dvipsnames]{pstricks}
\usepackage{epsfig}
\usepackage{pst-grad} 
\usepackage{pst-plot} 
\usepackage[space]{grffile} 
\usepackage{etoolbox} 
\makeatletter 
\patchcmd\Gread@eps{\@inputcheck#1 }{\@inputcheck"#1"\relax}{}{}
\makeatother
\usepackage{longtable}
\usepackage{tikz}
\usetikzlibrary{arrows}
\usetikzlibrary{decorations.pathreplacing}

\usepackage{longtable}

\numberwithin{equation}{section}

\newtheorem{theorem}{Theorem}[section]

\newtheorem{lemma}[theorem]{Lemma}
\newtheorem{corollary}[theorem]{Corollary}

\theoremstyle{definition}

\newcommand{\appsection}[1]{\let\oldthesection\thesection
\renewcommand{\thesection}{Appendix \oldthesection}
\section{#1}\let\thesection\oldthesection}

\newtheorem{definition}[theorem]{Definition}
\newtheorem{notation}[theorem]{Notation}

\theoremstyle{remark}

\newtheorem{remark}[theorem]{Remark}
\newtheorem{example}[theorem]{Example}

\DeclareMathOperator{\spec}{Spec}

\def\Z{{\mathbb{Z}}}
\def\Q{{\mathbb{Q}}}
\def\C{{\mathbb{C}}}
\def\P{{\mathbb{P}}}

\def\O{{\mathcal{O}}}

\def\L{{\mathcal{L}}}

\def\M{{\mathcal{M}}}

\def\N{{\mathcal{N}}}

\begin{document}
\title{Families of explicit quasi-hyperbolic and hyperbolic surfaces}
%
\author[Natalia Garcia-Fritz]{Natalia Garcia-Fritz}
\email{natalia.garcia@mat.uc.cl}
\address{Facultad de Matem\'aticas, Pontificia Universidad Cat\'olica de Chile, Campus San Joaqu\'in, Avenida Vicu\~na Mackenna 4860, Santiago, Chile.}
\author[Giancarlo Urz\'ua]{Giancarlo Urz\'ua}
\email{urzua@mat.uc.cl}
\address{Facultad de Matem\'aticas, Pontificia Universidad Cat\'olica de Chile, Campus San Joaqu\'in, Avenida Vicu\~na Mackenna 4860, Santiago, Chile.}
%
%

\maketitle


\date{\today}

\begin{abstract}
We construct explicit families of quasi-hyperbolic and hyperbolic surfaces. This is based on earlier work of Vojta, and the recent expansion and generalization of it by the first author. In this paper we further extend it to the singular case, obtaining results for the surface of cuboids, the generalized surfaces of cuboids, and other families of Diophantine surfaces of general type. In particular, we produce explicit families of smooth complete intersection surfaces of multidegrees $(m_1,\ldots,m_n)$ in $\P^{n+2}$ which are hyperbolic, for any $n \geq 8$ and any degrees $m_i \geq 2$. We also show similar results for complete intersection surfaces in $\P^{n+2}$ for $n=4,5,6,7$. These families give evidence for \cite[Conjecture 0.18]{Dem18} in the case of surfaces.
\end{abstract}

%

\section{Introduction} \label{s0}

The purpose of this paper is to give an explicit method to find low genus curves in a wide range of algebraic surfaces. The method is based on an earlier work of Vojta \cite{V00}, which has roots in the seminal work of Bogomolov \cite{B77} (see \cite{D79}), and the recent expansion and generalization of Vojta's method by the first author \cite{GF15}. In this paper we further extend it to the singular case.  

In addition, we show that the method allows us to test hyperbolicity on these surfaces. In particular, we show new examples of families of quasi-hyperbolic and hyperbolic surfaces. This part is based on Nevanlinna theory (cf.\ \cite{V11}). We recall some definitions to be precise. An \textit{entire curve} in a variety $X$ is the image of a nonconstant holomorphic map $\C \to X$. A surface $X$ is said to be \emph{quasi-hyperbolic} if all entire curves are contained in a proper Zariski closed subset of $X$. A surface $X$ is said to be \textit{hyperbolic} if it has no entire curves. Hence, when $X$ is a smooth projective surface, we have that $X$ is hyperbolic if it is in the sense of Kobayashi or in the sense of Brody; cf. \cite{Kob}.     

A main motivation for us comes from describing the set of rational points of particular Diophantine varieties under the Bombieri-Lang conjecture. For instance, by finding all curves of geometric genus less than or equal to one, Vojta \cite{V00} shows that the ``$n$ squares problem" of B\"uchi follows from that conjecture, and later the first author shows that the analogous problem for arbitrary $k$-powers would also be a consequence of it \cite{GF16}.


Let us consider one example, which will be used in Section \ref{s1} to develop the ideas and computations around the method. In this example the singularities are rational double points of type A.

Let $n\geq 3$, $m\geq 2$ be integers. Let $\{F_1,\ldots,F_{nm}\}$ and $\{G_1,\ldots, G_{nm}\}$ be collections of distinct $nm$ vertical and horizontal fibres of $\P^1 \times \P^1$ respectively. Let us denote the elements of Pic$(\P^1 \times \P^1) \simeq \Z^2$ by $(a,b)$. Then we have $$ F_{km+1}+\ldots+F_{(k+1)m} + G_{km+1}+\ldots+ G_{(k+1)m} = (m,m) $$ in Pic$(\P^1 \times \P^1)$ for $0\leq k\leq n-1$. These equations define a tower of $n$ cyclic covers of degree $m$ $$X_n \to X_{n-1} \to \ldots \to X_1 \to \P^1 \times \P^1.$$ All $X_k$ are normal projective surfaces with $k m^{k+1}$ singularities of type $$A_{m-1} \colon \ (0,0) \in (z^m-xy)\subset \C^3.$$ The surface $X_n$ is simply connected. It also has ample canonical class, and so it is of general type.

Let $n=3$. We can take as a model of $\P^1 \times \P^1$ the quadric $$(z_0z_3-z_1z_2) \subset \P^3,$$ and so the surface $X_3$ can be presented as $$ \prod_{i=1}^m(z_0-a_iz_1-b_iz_2+a_i b_iz_3) =z_4^m, \ \ \ \prod_{i=1}^m(z_0-c_iz_1-d_iz_2+c_i d_iz_3) =z_5^m, $$ 
$$ \prod_{i=1}^m(z_0-e_iz_1-f_iz_2+e_i f_iz_3) =z_6^m, \ \ \ \ \ z_0z_3-z_1z_2=0 $$ in $\P^6$ for distinct $a_i,c_i,e_i \in \C$ and distinct $b_i,d_i,f_i \in \C$. For $m=2$ and a specific choice of $a_i,c_i,e_i,b_i,d_i,f_i$, the surface $X_3$ is isomorphic to the surface of cuboids $S$ defined by  $$x_0^2+x_1^2+x_2^2=x_3^2, \ \ \ \ x_0^2+x_1^2=x_4^2, \ \ \ \ x_0^2+x_2^2=x_5^2, \ \ \ \ x_1^2+x_2^2=x_6^2$$ in $\P^6$ (see Section \ref{s2}). A positive rational point in $S$ would realize a perfect cuboid. It is unknown if there are any such points. This famous old problem goes back to Euler; cf.\ \cite{ST10}, \cite{vLu00}, \cite{B13}, \cite{FS16}. According to the \textit{Bombieri-Lang conjecture}, outside of a certain finite set of curves of geometric genus at most one, 
there can only be a finite number of solutions. Hence it is of interest to find all such curves. We prove the following (see Sections \ref{s1} and \ref{s5}).

\begin{theorem}
The surfaces $X_n$ are hyperbolic for any $m>2$.
\end{theorem}     
 
Recall that the surface of cuboids $S$ is the surface $X_n$ with $m=2$, $n=3$ and a particular choice of vertical and horizontal fibres. For $S$, the results are less strong and involve further adaptations of the method. Details are worked out in Section \ref{s2}.

\begin{theorem}
Let $S$ be the surface of cuboids. Let $S' \to S$ be its minimal resolution, and let $E$ be the sum of the $48$  exceptional curves. We have: 
\begin{itemize}
\item[(a)] Every curve of geometric genus $0$ or $1$ must contain at least $2$ of the $48$ singularities of $S$.

\item[(b)] If $C \subset S$ is a curve which is smooth at the singular points of $S$, then deg$(C) \leq 4 g(C)+44$.

\item[(c)] If $C \subset S'$ is a rational curve which is neither exceptional nor contained in $x_0 x_1 x_2 x_3=0$, then $C \cdot E \geq 8$. 

\end{itemize}
\end{theorem} 

In Section \ref{s3} we develop the method for arbitrary cyclic quotient singularities, which is captured in the following particular example. Let us consider the lines $ L_{t,u} = (t^2x + tuy + u^2 z) \subset \P^2$ for $[t,u] \in \P^1$. They are precisely the tangent lines to the conic $ (y^2-4xz) \subset \P^2$. Let $\{L_1,\ldots,L_{d} \}$ be distinct lines such that $L_i=L_{t_i,u_i}$ for some $[t_i,u_i] \in \P^1$. Let us take positive integers $a_1,\ldots,a_d$ such that $\sum_{i=1}^d a_i =mR$ for some integers $m,R>0$. Assume that $a_i<m$ and gcd$(a_i,m)=1$ for all $i$, and that $a_i+a_j$ is not divisible by $m$ for all $i \neq j$. The method we develop in Section \ref{s3} allows us to prove the following.

\begin{theorem}
If $4m < d$, then the surface $$(t_1^2x+t_1 u_1 y+u_1^2z)^{a_1}\cdots(t_d^2x+t_d u_d y+u_d^2z)^{a_d}=w^m$$ in $\P(1,1,1,R)$  contains no curves with geometric genus $\leq 1$ apart from the $\P^1$'s defined by $t_i^2x+t_i u_i y+u_i^2z=0$. Its normalization is a simply connected normal projective surface with ample canonical class.
\label{hshs}  
\end{theorem}

The Diophantine hypersurface $(\prod_{i=1}^{15} (i^2x+iy+z)= w^3) \subset \P(1,1,1,5)$ works as an example for Theorem \ref{hshs}. 

Let $X$ be a smooth projective surface, and let $\omega \in H^0(X, \L \otimes S^r\Omega_X^1)$ for some line bundle $\L$, and some integer $r>0$. Key in the above results is the notion of $\omega$-integral curves; see Definition \ref{omega}. We construct surfaces via a composition of cyclic covers of the fixed surface $X$, which are branched along $\omega$-integral curves, and for which we know all $\omega$-integral curves on $X$. In Section \ref{s4} we prove the following general theorem. 

\begin{theorem}\label{main}
Assume we have the relations $m_i \M_i = \sum_{j=1}^{s_i} a_{i,j} D_{i,j}$ in Pic$(X)$, for some line bundles $\M_i$, with $0<a_{i,j}<m_i$ and gcd$(a_{i,j},m_i)=1$ for all $i,j$. Assume  that the divisor $\sum_{i=1}^n \sum_{j=1}^{s_i} D_{i,j}$ has simple normal crossings with $D_{i,j}$ $\omega$-integral curves.  Then a tower of $n$ cyclic covers of degree $m$ $$X_n \to X_{n-1} \to \ldots \to X_1 \to X_0:= X$$ is defined, where all $X_k$ are normal projective surfaces with only cyclic quotient singularities. If $$\sum_{i=1}^n \frac{1} {m_i} \Big( \sum_{j=1}^{s_i} D_{i,j}\Big) -\L \ \ \text{is} \ \Q\text{-ample, and } \ a_{i,j} \not \equiv -a_{i,j'} (mod \ m_i) \ \text{ for all} \ i, \ j\neq j',$$ then $X_n$ can have curves of geometric genus $\leq 1$ only in the set of preimages of $\omega$-integral curves in $X$. 
\end{theorem}

From this theorem, and via a particular result of Vojta \cite{V00b} in Nevanlinna theory for algebraic varieties, in Section \ref{s5} we obtain the following. 

\begin{theorem}
Let us consider the hypothesis and notation as in Theorem \ref{main}. In addition, assume that all solutions to the differential equation given by $\omega=0$ on $X$ are $\omega$-integral curves. Then an entire curve in $X_n$ must be contained in the set of preimages of $\omega$-integral curves in $X$. In particular, if the set of preimages of $\omega$-integral curves in $X$ does not contain curves of geometric genus $0$ or $1$, then $X_n$ is hyperbolic.
\end{theorem}

The following application of the method and its adaptations produces explicit families of smooth complete intersection surfaces which are hyperbolic and have arbitrary multidegrees, in particular, low multidegrees. We note that it is key for this application to work with singularities.

\begin{theorem}
There are explicit families of smooth complete intersections in $\P^{n+2}$ which are hyperbolic for multidegrees

\begin{itemize}
\item[(a)] $(m_1,\ldots,m_n)$ when $n\geq 8$ and $m_i\geq 2$.

\item[(b)] $(m_1,\ldots,m_n)$ when $n\geq 5$ and $m_i\geq 3$.

\item[(c)] $(2,m_1,\ldots,m_{n-1})$ when $n \geq 6$, $m_i \geq 2$ for all $i<n-1$, $m_{n-1} \geq 3$.

\item[(d)] $(2,m_1,\ldots,m_{n-1})$ when $n\geq 4$ and $m_i \geq 3$.
\end{itemize} 
\label{jdjhds}
\end{theorem}

The explicit families are shown at the end of Section \ref{s5}. This theorem gives evidence for \cite[Conjecture 0.18]{Dem18} in the case of surfaces. Smooth complete intersection surfaces $X_{n,k} \subset \P^n$ of multidegree $(k,\ldots,k)$ and hyperbolic have been constructed in \cite{GF16} for $k=3$, $n\geq 6$; $k=4,5$, $n\geq 5$; $k \geq 6$, $n\geq 4$. We also point out that hyperbolic complete intersections of high multidegree have been constructed by Brotbek \cite{Br14} (see also \cite{X15}).

\subsection*{Acknowledgements} We are grateful to Jean-Pierre Demailly, Simone Diverio, Bruno de Oliveira, and Damiano Testa for interesting conversations and email correspondence. The first author is supported by the FONDECYT Iniciaci\'on en Investigaci\'on grant 11170192, and the CONICYT PAI grant 79170039. The second author was supported by the FONDECYT regular grant 1150068.

\section{The generalized surfaces of cuboids} \label{s1}

Let $n\geq 3$, $m\geq 2$ be integers. Let $\{F_1,\ldots,F_{nm}\}$ and $\{G_1,\ldots, G_{nm}\}$ be collections of distinct $nm$ vertical and horizontal fibres of $\P^1 \times \P^1$. Let us denote the elements of Pic$(\P^1 \times \P^1) \simeq \Z^2$ by $(a,b)$. Then we have $$ F_{km+1}+\ldots+F_{(k+1)m} + G_{km+1}+\ldots+ G_{(k+1)m} = (m,m) $$ for $0\leq k\leq n-1$. These expressions define a tower of $n$ cyclic covers of degree $m$ $$X_n \to X_{n-1} \to \ldots \to X_1 \to X_0:= \P^1 \times \P^1 $$ inductively as follows: Let $f_{k+1} \colon X_{k+1} \to X_{k}$ be the cyclic cover defined by the equation of line bundles $$ g_k^* \big(F_{km+1}+\ldots+F_{(k+1)m} + G_{km+1}+ \ldots +G_{(k+1)m} \big) \simeq g_k^*(1,1)^{\otimes m},$$ where $g_k:= f_1 \circ \cdots \circ f_k$. To be more precise, the surface $X_{k+1}$ is defined as $X_{k+1}:= \spec_{X_k} \bigoplus_{j=0}^{m-1} g_k^*(1,1)^{-j},$ and so the finite morphism $f_{k+1}$ satisfies ${f_{k+1}}_* \O_{X_{k+1}} =\bigoplus_{j=0}^{m-1} g_k^*(1,1)^{-j}$; for details see e.g.\ \cite[Section 1]{U10}. All $X_k$ are normal projective surfaces with $k m^{k+1}$ singularities of type $$A_{m-1} \colon \ (0,0) \in (z^m-xy)\subset \C^3.$$ As shown in the introduction for $n=3$, the surfaces $X_k$ are complete intersections in $\P^{k+3}$, and so they are simply connected.

Let $D_{k} \subset X_{k-1}$ be the branch divisor of $f_k \colon X_k \to X_{k-1}$, this is $$D_k:= g_{k-1}^*\big(F_{km+1}+\ldots+F_{(k+1)m} + G_{km+1}+ \ldots +G_{(k+1)m} \big).$$ Hence it is a collection of $2m$ smooth curves which form a simple normal crossings divisor with $m^{k+1}$ nodes. By the Riemann-Hurwitz formula, each smooth curve $\Gamma$ in $D_k$ satisfies $$ 2g(\Gamma)-2=m^{k-1}((m-1)(k-1)-2),$$ and $\Gamma^2=0$, where $g(\Gamma)$ is the genus of $\Gamma$. We have $D_k^2=2m^{k+1}$, and $D_k \cdot K_{X_{k-1}}=2m(2g(\Gamma)-2)$. We also have the formulas $$\chi(\O_{X_k})=m \chi(\O_{X_{k-1}}) + \frac{(m-1)(2m-1)}{12m} D_k^2 + \frac{(m-1)}{4} D_k \cdot K_{X_{k-1}},$$ and $ K_{X_k}^2= m K_{X_{k-1}}^2 + \frac{(m-1)^2}{m} D_k^2 + 2(m-1) D_k \cdot K_{X_{k-1}},$ and so $$K_{X_k}^2 -8\chi(\O_{X_k})= -\frac{2k}{3} m^k(m^2-1) <0.$$ In fact $ \frac{K_{X_k}^2}{\chi(\O_{X_k})}$ approaches $8$ as $k>>0$. We also have $$K_{X_k} \sim g_k^*\big((-2,-2) + \frac{(m-1)}{m} (m,m)k)\big)=(k(m-1)-2) g_k^*(1,1),$$ and so $X_k$ has ample canonical class if and only if $k(m-1)>2$. Hence for $k=n$ the surface $X_n$ is of general type.

Let $\sigma_n \colon X_n' \to X_n$ be the minimal resolution of the singularities in $X_n$, and let $g_n'=\sigma_n \circ g_n$. Then $${g_n'}^*(mn,mn)=mR+mE$$ where $R$ is the strict transform by $\sigma_n$ of the branch divisor (in $X_0$) of $g_n$, and $E$ is the (reduced) sum of the exceptional curves of $\sigma_n$. That is a simple local toric computation; see e.g.\ \cite[2.1]{HTU17}. Therefore, we have\begin{equation}
{g_n'}^*(n,n)=R+E.
\label{eq1}
\end{equation}
We now recall the key general notion of $\omega$-integral curve (see \cite[Definition 2.4]{V00}, \cite[Definition 3.2]{GF15}). 

\begin{definition}
Let $X$ be a smooth surface. Let $r\geq 1$ be an integer, let $\L$ be an invertible sheaf on $X$, and let $\omega \in H^0(X,\L \otimes S^r\Omega_X^1)$. Let $C$ be an irreducible curve on $X$, and let $\varphi_C \colon \tilde{C} \to X$ be the normalization of $C \subset X$. The curve $C$ is said to be $\omega$\emph{-integral} if $\varphi_C^* \omega \in H^0(\tilde{C}, \varphi_C^* \L \otimes S^r \Omega_{\tilde{C}}^1)$ is zero.
\label{omega}
\end{definition}

Let $\omega \in H^0(X_0, (2,2) \otimes S^2 \Omega^1_{X_0})$ be the global section $z_3^2 dz_1 dz_2$. Consider the isomorphism $h \colon \P^1 \times \P^1 \to X_0$ given by $h([x,y] \times [w,z])=[xw,xz,yw,yz]$. The section $\omega$ corresponds to the section $y^2z^2 dx dw$ under $h$. Therefore, horizontal and vertical fibres are $\omega$-integral, and this is the complete set of $\omega$-integral curves by \cite[6.6]{GF15}. In particular the branch loci of $g_n \colon X_n \to X_0$ is formed by $\omega$-integral curves.

\begin{notation}
Let $X, r, \L$, and $\omega$ be as in Definition \ref{omega}. Let $Y$ be a smooth surface, and $\pi \colon Y \to X$ be a dominant morphism. We denote by $\pi^{\bullet} \omega$ the image of $\omega$ under the natural pull-back morphism $H^0(X,\L \otimes S^r\Omega_X^1) \to H^0(Y,\pi^*(\L) \otimes S^r\Omega_Y^1)$; see \cite{GF15}.
\label{punto}
\end{notation}

As in \cite[Theorem 3.87]{GF15}, we have that there exists a section $\omega'$ in $H^0(X_n', \O_{X_n'} (-(m-1)R) \otimes g_n'^*(2,2) \otimes S^2\Omega_{X_n'}^1)$ whose image under the natural morphism $$H^0\big(X_n', \O_{X_n'} (-(m-1)R) \otimes g_n'^*(2,2) \otimes S^2\Omega_{X_n'}^1\big) \to H^0\big(X_n', g_n'^*(2,2) \otimes S^2\Omega_{X_n'}^1\big)$$ is precisely $g_n'^{\bullet} \omega$. We now show the existence of a section $\omega''$ in a ``more negative" sheaf, which also has $g_n'^{\bullet} \omega$ as image.

\begin{lemma}
Let $m>2$. Then there is $\omega''$ in $H^0(X_n', \O_{X_n'} (-(m-1)R-E) \otimes g_n'^*(2,2) \otimes S^2\Omega_{X_n'}^1)$ whose image under the natural morphism $$H^0\big(X_n', \O_{X_n'} (-(m-1)R-E) \otimes g_n'^*(2,2) \otimes S^2\Omega_{X_n'}^1\big) \to H^0\big(X_n', g_n'^*(2,2) \otimes S^2\Omega_{X_n'}^1\big)$$ is $g_n'^{\bullet} \omega$.
\label{newomega}
\end{lemma}

\begin{proof}
We only need to prove that $g_n'^{\bullet} \omega$ vanishes along the divisor $E$. This is a general toric local computation with differentials. So let us say that $\omega \in H^0(X_0, \L \otimes S^r\Omega_{X_0}^1)$, where in our case $\L=(2,2)$ and $r=2$.

Let us consider one node $P$ of the branch divisor of $g_k$ for some $0<k<n$. Take one preimage $Q$ of this node $P$ by the morphism $g_n$. At $Q$ we have a rational double point of type $A_{m-1}$, which is in particular a cyclic quotient singularity. By \cite[III, Theorem 5.1]{BHPV04}, the singularity and the map $g_n \colon X_n \to X_0$ is locally analytically isomorphic to $U=(z^m-xy) \subset \C^3$ and the projection $g(x,y,z)=(x,y)$ respectively. Thus to compute the pull-back of $\omega$ under $g_n$, we use this local model. Moreover, this model has a toric description as follows. See e.g \cite{R03}.

Let $\sigma' \colon V \to U$ be the minimal resolution, and $g' \colon V \to \C^2$ the composition of $\sigma'$ with $g$. Let $E_0$ be the (reduced) preimage under $g'$ of $x=0$, and let $E_m$ be the (reduced) preimage under $g'$ of $y=0$. Let $E_1, E_2, \ldots, E_{m-1}$ be the chain of $\P^1$'s in $V$ which corresponds to the exceptional divisor of $\sigma'$. Hence $E_0,E_1,\ldots,E_{m-1},E_m$ also form a chain. Let $u_i=0$ be the local coordinate defining $E_i$, so that at each node of $E_0,E_1,\ldots,E_{m-1},E_m$ we have local coordinates $u_i, u_{i+1}$ for $V$. Then (see \cite[Example 3.1]{R03}) we have that locally $g'$ is given by $$g'(u_i,u_{i+1})=(u_i^{m-i}u_{i+1}^{m-i-1},u_i^{i}u_{i+1}^{i+1}).$$ Therefore we have the pull-back relations $$dx=(m-i)u_i^{m-i-1}u_{i+1}^{m-i-1} du_{i}+(m-i-1)u_i^{m-i}u_{i+1}^{m-i-2} du_{i+1}$$ and $$dy=iu_i^{i-1}u_{i+1}^{i+1} du_{i}+(i+1)u_i^{i}u_{i+1}^{i} du_{i+1}.$$

Via a local trivialization of $\L$, we can identify $\omega$ with a section of $S^r \Omega_{\C^2}^1$ around $(0,0)$ as $$ \omega= a_0 dx^{\otimes r} + a_1 dx^{\otimes (r-1)} \otimes dy + \ldots+ a_r dy^{\otimes r},$$ where the $a_i$ are holomorphic around $(0,0)$. Since $x=0$ and $y=0$ are $\omega$-integral curves, then as in \cite[Theorem 3.87]{GF15}, we have that $a_0=y a_0'$ and $a_r=x a_r'$. Therefore the pull-back of $a_0 dx^{\otimes r}$ and $a_r dy^{\otimes r}$ vanish along the divisor $E_1+\ldots+E_{m-1}$. On the other hand, the pull-back of $dx^{\otimes (r-k)} \otimes dy^{\otimes k}$ for $0<k<r$ vanishes on $u_i=0$ if and only if $i-1>0$ or $m-i-1>0$, and both inequalities hold because $m>2$. 
Therefore the pull-back of $\omega$ vanishes along $E_0+E_1+\ldots+E_{m-1}+E_{m}$. Moreover, by \cite[Theorem 3.87]{GF15}, it vanishes of order $m-1$ along $E_0$ and $E_m$.    
\end{proof}

\begin{remark}
The previous lemma is valid for any tower of cyclic morphisms of order $m>2$ branched along a simple normal crossings divisor which is formed by $\omega$-integral curves, so that the singularities are rational double points of type $A_{m-1}$. 
\end{remark} 

\begin{theorem}
Let $m>2$, and let $C \subset X_n'$ be a curve of geometric genus $g$ which is not an exceptional curve of $\sigma_n$. If $$\frac{4g-4}{n-2} < g_n'(C) \cdot (1,1),$$ then $g_n'(C)$ is an $\omega$-integral curve.
\label{gencub}
\end{theorem}

\begin{proof}
By Lemma \ref{newomega}, we have $$\omega'' \in H^0(X_n', \O_{X_n'} (-(m-1)R-E) \otimes g_n'^*(2,2) \otimes S^2\Omega_{X_n'}^1)$$ whose image under the natural morphism $$H^0\big(X_n', \O_{X_n'} (-(m-1)R-E) \otimes g_n'^*(2,2) \otimes S^2\Omega_{X_n'}^1\big) \to H^0\big(X_n', g_n'^*(2,2) \otimes S^2\Omega_{X_n'}^1\big)$$ is $g_n'^{\bullet} \omega$. Let $\varphi_C \colon \tilde{C} \to X_n'$ be the normalization of $C \subset X_n'$. By Equality \eqref{eq1}, we obtain that $$g_n'^*(-n,-n)-(m-2)R = -(m-1)R-E$$ in Pic$(X_n')$. In this way $$\O_{X_n'}(-(m-1)R-E) \otimes g_n'^*(2,2) \simeq g_n'^*(2-n,2-n) \otimes \O_{X_n'}(-(m-2)R)=: \L, $$ and so $\deg_{\tilde{C}} \Big(\varphi_C^* \L \otimes S^2\Omega_{\tilde{C}}^1 \Big) \leq (2-n,2-n) \cdot g_n'(C) + 2(2g-2) <0 $ by the projection formula and the hypothesis. Therefore $H^0(\tilde{C}, \varphi_C^* \L \otimes S^2 \Omega_{\tilde{C}}^1)=0$, and $C$ is a $\omega''$-integral curve. As in \cite[Proposition 3.88]{GF15}, the $\omega''$-integral curves in $X_n'$ are also $g_n'^{\bullet} \omega$-integral curves. On the other hand, by \cite[Theorem 3.35]{GF15} the $g_n'^{\bullet} \omega$-integral curves in $X_n'$ are either the exceptional divisors of $\sigma_n$ or curves $C \subset X_n'$ such that $g_n'(C)$ is $\omega$-integral in $X_0$. 
\end{proof}

\begin{corollary}
There are no curves of geometric genus $\leq 1$ in $X_n$ for any $m>2$.
\label{norateli}
\end{corollary}

\begin{proof}
By Theorem \ref{gencub}, if $C$ is a curve in $X_n'$ of geometric genus $0$ or $1$ which is not an exceptional curve of $\sigma_n \colon X_n' \to X_n$, then $g_n'(C)$ is $\omega$-integral. All $\omega$-integral curves are fibres of $\P^1 \times \P^1$. But, since $n>2$ and $m>2$, we have that the preimage of a fibre in $\P^1 \times \P^1$ has geometric genus bigger than $1$. Therefore the only curves of geometric genus $0$ or $1$ in $X_n'$ are the exceptional curves. As $\sigma_n$ contracts them, $X_n$ has no such curves. 
\end{proof}

\begin{remark}
Theorem \ref{gencub} is not true for $m=2$, since otherwise we would obtain all curves with geometric genus $\leq 1$ from fibres, and that is not the case (see Section \ref{s2}). The problem is that Lemma \ref{newomega} does not work for $m=2$, showing optimality in that sense. We will revisit this issue in Section \ref{s3}.
\end{remark}

\begin{example}
Under the Bombieri-Lang conjecture, Corollary \ref{norateli} says that, for example, the complete intersection surface
$$\prod_{i=1}^m (x_0-ix_1+i^2x_2)=x_4^m \ \ \ \ \ \ \prod_{i=1}^m (x_0-(i+m)x_1+(i+m)^2x_2)=x_5^m $$ $$\prod_{i=1}^m (x_0-(i-m)x_1+(i-m)^2x_2)=x_6^m \ \ \ \ \ \ \ x_0 x_3+ x_2^2=x_1 x_2$$ in $\P^6$ can only have a finite number of points in $\P^6(\Q)$. Here $n=3$, and we have chosen a specific model and set of parameters in $\Z$.
\end{example}

\section{The surface of cuboids} \label{s2}

We know that Theorem \ref{gencub} does not work for the surface of cuboids $$z_0z_3=z_4^2 \ \ \ \ \ (z_0-z_3)^2+(z_1+z_2)^2=z_5^2 $$ $$(z_0+z_3)^2-(z_1+z_2)^2=z_6^2 \ \ \ \ \ z_0 z_3=z_1 z_2 $$ in $\P^6$, since here $m=2$. The purpose of this section is to adapt the method to get some results on rational curves of this surface. We follow the notation of Section \ref{s1} for this particular example, and so the surface of cuboids is denoted by $X_3$. We recall that $X_3$ is isomorphic to the original surface of cuboids: $$x_0^2+x_1^2+x_2^2=x_3^2, \ \ \ \ x_0^2+x_1^2=x_4^2, \ \ \ \ x_0^2+x_2^2=x_5^2, \ \ \ \ x_1^2+x_2^2=x_6^2$$ in $\P^6$, by the isomorphism $z_0=x_0-ix_1$, $z_1=x_3-x_2$, $z_2=x_3+x_2$, $z_3=x_0+ix_1$, $z_4=x_4$, $z_5=2x_5$, $z_6=2i x_6$, where $i=\sqrt{-1}$. This surface has been extensively studied, because it is related to the Perfect Cuboid Problem of Euler. Some references are \cite{ST10}, \cite{vLu00}, \cite{B13}, \cite{FS16}.

There are at least $92$ curves of geometric genus zero or one in $X_3$. They are all smooth (see \cite{vLu00}):

\begin{itemize}
\item The irreducible components of $x_0x_1x_2x_3=0$, which are $32$ rational curves;
\item The irreducible components of $x_4x_5x_6=0$, which are $12$ elliptic curves, corresponding to the pull-backs of the $6$ horizontal fibres $\{F_1,\ldots, F_6\}$, and the $6$ vertical fibres $\{G_1,\ldots,G_6\}$; 
\item The curve defined by the equations 
$$x_0=x_1, \ \ x_4=x_5, \ \ \sqrt{2}x_0 = x_6, \ \ x_2^2+x_6^2=x_3^2, \ \ 2x_5^2+x_6^2=2x_3^2,$$
and the curves obtained as orbits by applying the automorphisms of $X_3$. This gives us $48$ elliptic curves.
\end{itemize}

It was proved by Stoll and Testa \cite{ST10} that every curve of geometric genus $\leq 1$ in $X_3$ of degree less than or equal to $4$ belongs to this list, and they conjectured that these are all the curves with geometric genus $\leq 1$ in this surface.

Using the global section (appearing in Section \ref{s1}) $$\omega\in H^0(X_0,(2,2)\otimes S^2\Omega^1_{X_0}),$$ from a weaker version of Lemma 2.3 that works for the case $m=2$, it is obtained the following (cf.\  Corollary 6.43 and Theorem 6.48 in \cite{GF15}):

\begin{theorem}
Let $C$ be an irreducible curve on $X_3$ with strict transform $C'\subseteq X'_3$. If
$$\mathrm{deg}_{\tilde{C}}(\varphi_C^*(\mathcal{O}_{X_3}(-R+E)\otimes g_3'^*(2,2))\otimes S^2\Omega^1_{\tilde{C}})=-\mathrm{deg}(C)+(E.C')+4g(C)-4$$
is negative, then $g_n'(C)$ is an $\omega$-integral curve.
\end{theorem}

This result is not enough to capture all the curves of geometric genus $\leq 1$ on $X_3$, but it allows us to give extra information about these curves, as the following corollaries show (cf.\ Proposition 1.11 and Corollary 1.13 in \cite{GF15}):


\begin{corollary}
Every curve of geometric genus zero or one on $X_3$ contains at least two of the $48$ singular points of $X_3$.
\end{corollary}

It was known by work of Freitag and Salvati Manni \cite{FS16} (from \cite{B13}) that a curve of geometric genus zero or one must contain at least one of the $48$ singular points of $X_3$.

\begin{corollary}
Let $C$ be an irreducible curve in $X_3$, smooth at the singularities of this surface ($C$ can have singularities outside of the $48$ singular points of $X_3$). Then
$$\mathrm{deg}(C)\leq  4g(C)+44.$$
\end{corollary}

This was also obtained by Kani for smooth curves \cite{K14}, using different methods, and it improves a result of Freitag and Salvati Manni from \cite{FS16}.
In this work we want to improve these results, by using different global twisted differentials at the same time, in order to get better control on the exceptional divisors.

Recall the tower of cyclic covers of degree $m=2$
$$X_3\to X_2\to X_1\to \mathbb{P}^1\times \mathbb{P}^1.$$
In this case $R$ consists of the horizontal fibres at the points $$[1:1],\ [1:-1],\ [1:i],\ [1:-i],\ [1:0],\ [0:1],$$ and the vertical fibres at the same points. We will denote the horizontal fibre at $[a:b]$ by $h_{b/a}$ and the vertical fibre at $[a:b]$ by $v_{b/a}$, with the convention that $1/0=\infty$. 

The image of $x_4=0$ in $\mathbb{P}^1\times\mathbb{P}^1$ consists of $h_{1}\cup h_{-1}\cup v_{1}\cup v_{-1}$, the image of $x_5=0$ consists of $h_{i}\cup h_{-i}\cup v_i\cup v_{-i}$, and the image of $x_6=0$ consists of $h_0\cup h_\infty\cup v_0\cup v_\infty$.
Thus, each of these consists of two horizontal fibres and two vertical fibres, intersecting at $4$ points. Over each of the $12$ intersection points, there are $4$ out of the $48$ singular points of $X_3$, and every singular points of $X_3$ maps to one of these intersections.


The image of $x_0=0$ under $g_3$ is the curve $C_0=\{xz+yw=0\}.$ Similarly, the image of $x_1=0$ is $C_1=\{yw-xz=0\}$, the image of $x_2=0$ is $C_2=\{yz-xw=0 \}$, and the image of $x_3=0$ is $C_3=\{xw+yz=0\}$.

Consider the following global sections in $H^0(\mathbb{P}^1\times\mathbb{P}^1,(3,3)\otimes S^2\Omega^1_{\mathbb{P}^1\times\mathbb{P}^1})$: $$\omega_0=(xz-yw)y^2z^2dxdw \ \ \ \ \ \ \ \ \omega_1=(yw-xz)y^2z^2dxdw $$ $$\omega_2=(yz-xw)y^2z^2dxdw \ \ \ \ \ \ \ \omega_3=(xw+yz)y^2z^2dxdw.$$ 

For each $0\leq i\leq 3$, the $\omega_i$-integral curves consist of the horizontal fibres, the vertical fibres, and $C_i$.

Let $E_i$ be the sum of the exceptional divisors from the $24$ singular points of $X_3$ whose image belongs to $C_i$, and let $E_i'$ be the sum of the exceptional divisors from the $24$ singular points whose image does not belong to $C_i$. We have the following version of Lemma \ref{newomega} adapted to these new global sections.

\begin{lemma}
For each $0\leq i\leq 3$, there is $\omega_i''$ in $$H^0(X_3', \O_{X_3'} (-R-E_i) \otimes g_3'^*(3,3) \otimes S^2\Omega_{X_3'}^1)$$ whose image under the natural morphism $$H^0\big(X_3', \O_{X_3'} (-R-E_i) \otimes g_3'^*(3,3) \otimes S^2\Omega_{X_3'}^1\big) \to H^0\big(X_3', g_3'^*(3,3) \otimes S^2\Omega_{X_3'}^1\big)$$ is $g_3'^{\bullet} \omega_i$.
\label{newomega2}
\end{lemma}

\begin{proof} Fix $0\leq i\leq 3$. We know that $g_3'^\bullet\omega_i$ vanishes along $R$. We will prove that $g_3'^{\bullet} \omega$ vanishes along the divisor $E_i$. 

We consider a node $P$ of the branch divisor of $g_k$ for some $0<k<3$, and such that $P\in g_k'(C_i)$. At the preimage $Q$ of $P$, we have a singularity of type $A_1$.
This singularity is locally analytically isomorphic to $(z^2-xy)\subset\C^3$ and $g(x,y,z)=(x,y)$. 

Let $E$ be the exceptional divisor at $Q$, let $u_1=0$ be the local coordinate defining it, and, as before, the local coordinates $u_0=0$, $u_2=0$ define the (reduced) preimages of $x=0$ and $y=0$ respectively. Then we have $$dx=2u_0u_1du_0+u_0^2du_1, \ \ \ \ \ \ dy=u_2^2du_1+2u_1u_2du_2.$$



Via a local trivialization of ${g'}_3^*(3,3)$, we can identify $\omega_0$ with a section of $S^2 \Omega_{\C^2}^1$ around $(0,0)$ as $\omega_0=(x-y)\bar{\omega}$ with $\bar{\omega}=a_0 dx^{\otimes 2} + a_1 dx dy + a_2 dy^{\otimes 2}$. Since ${g'}_3^*(x) =u_0^2u_1$, and ${g'}_3^*(y)=u_1u_2^2$, we obtain that the pull-back of $\omega_0$ vanishes along $E$ with order one. Doing this for every singular point contained in $C_0$, we obtain that ${g'}_3^{\bullet} \omega_0$ vanishes along the divisor $E_0$. A similar computation shows that for every $i$, the pull-back of $\omega_i$ vanishes along the divisor $E_i$.
\end{proof}


\begin{theorem}
Let $C\subset X_3'$ be a curve of geometric genus $0$, which is not an exceptional curve of $\sigma_3$, and it is not in the pull-back of the $C_i$'s. Then for every $0\leq i\leq 3$, we have $(C.E_i') \geq 4$.
\end{theorem}

\begin{proof}
From Proposition 3.88 in \cite{GF15}, we have that the $\omega_i''$-integral curves in $X_3'$ are among the pull-back of the horizontal fibres, the pull-back of the vertical fibres and the curve $C_i$. 
Let $C\subset X_3'$ be a curve of geometric genus zero. We have 
$$\mathcal{O}_{X'_3}(-R-E_i)\otimes g_3'^*(3,3)=\mathcal{O}_{X_3'}(-R-E_i+R+E)=\mathcal{O}_{X_3'}(E_i'),$$
thus if for some $0\leq i\leq 3$ we have $(C.E_i)<4$, then we obtain
$$\deg_{\tilde{C}}(\varphi_C^*\mathcal{O}_{X'_3}(-R-E_i)\otimes\varphi_C^*g_3'^*(3,3)\otimes S^2\Omega^1_{X_3'})=(C.E_i')-4<0,$$
hence $C$ must be an $\omega_i''$-integral curve.
\end{proof}


\begin{corollary}
Let $C\subset X_3'$ be a curve of geometric genus $0$, which is not an exceptional curve of $\sigma_3$, and not in the pull-back of the $C_i$'s. Then $(C.E) \geq 8$.
\end{corollary}

\begin{proof}
We know that $(C.E_i')\geq 4$ for each $0\leq i\leq 3$. Since $E_0'+E_1'+E_2'+E_3'=2E$, we obtain $(C.E)\geq 8$.
\end{proof}

\section{Low genus curves in cyclic covers} \label{s3}

\subsection{Local picture} \label{s31} We first recall the local picture of a cyclic cover, together with cyclic quotient singularities, and their minimal resolution.

Let $0<q<m$ be integers with gcd$(q,m)=1$. Consider the action of $\tau(x,y)=(\mu x, \mu^q y)$ on $\C^2$, where $\mu$ is a primitive $m$-th root of $1$. A \textit{cyclic quotient singularity} $\frac{1}{m}(1,q)$ is a germ at the origin of the quotient of $\C^2$ by $\langle \tau \rangle$; cf.\ \cite[III \S5]{BHPV04}. For us it will be useful the following toric description. Consider the inclusions of rings $$\C[x^m,y^m] \subset \C[x^m,y^m,x^{m-q}y] \subset \C[x,y]^{\langle \tau \rangle} \subset \C[x,y].$$ We note that $$\C[x^m,y^m,x^{m-q}y] \simeq \C[u,v,w]/(u v^{m-q}-w^m),$$ where $v=x^m$, $u=y^m$, and $w=x^{m-q}y$. The inclusions define morphisms between the corresponding spectrums of the rings, which translates into the maps $$\C^2 \xrightarrow{q} \C^2/\langle \tau \rangle \xrightarrow{\eta} (u v^{m-q}-w^m) \subset \C^3 \xrightarrow{r} \C^2$$ where $r(u,v,w)=(u,v)$ is the cyclic cover branch along $\{ u v^{m-q}=0 \}$ of degree $m$, $\eta$ is the normalization map, and $q$ is the quotient map. As in \cite[III \S5]{BHPV04}, around $(0,0) \in \C^2$ the local picture for any cyclic cover of degree $m$ is given by $\{ u^av^b=w^m \} \subset \C^3 \to \C^2$, $(u,v,w) \mapsto (u,v)$, where gcd$(a,m)$=gcd$(b,m)=1$, and $q$ is such that $aq+b \equiv 0$ modulo $m$.

Let $\sigma \colon \widetilde{Y} \rightarrow Y$ be the minimal resolution of $Y:=\frac{1}{m}(1,q)$. Figure \ref{exdiv} shows the exceptional curves $E_i=\P^1$ of $\sigma$, for $1 \leq i \leq s$, and the strict transforms $E_0$ and $E_{s+1}$ of $(y=0)$ and $(x=0)$ respectively.

\begin{figure}[htbp]
\includegraphics[width=11.5cm]{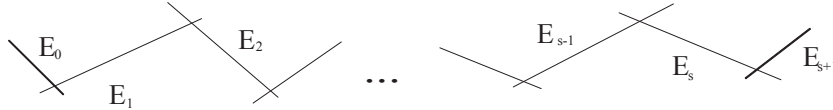}
\caption{Exceptional divisors over $\frac{1}{m}(1,q)$, $E_0$ and
$E_{s+1}$} \label{exdiv}
\end{figure}

The numbers $E_i^2=-b_i$ are computed using the {\em Hirzebruch-Jung continued fraction}
$$ \frac{m}{q} = b_1 - \frac{1}{b_2 - \frac{1}{\ddots - \frac{1}{b_s}}} =: [b_1, \ldots ,b_s].$$

The continued fraction $[b_1,\ldots,b_s]$ defines the sequence of integers $$ 0=\beta_{s+1} < 1=\beta_s < \ldots < q=\beta_1 < m= \beta_0 $$ where $\beta_{i+1}= b_{i}\beta_i - \beta_{i-1}$. In this way, $\frac{\beta_{i-1}}{\beta_{i}}=[b_i,\ldots,b_s]$. Partial fractions $\frac{\alpha_i}{\gamma_i} =[b_1,\ldots,b_{i-1}]$ are computed through the sequences $$ 0=\alpha_0 < 1=\alpha_1 < \ldots < q^{-1}=\alpha_s < m= \alpha_{s+1},$$ where $\alpha_{i+1}=b_i\alpha_{i} - \alpha_{i-1}$ ($q^{-1}$ is the integer such that $0<q^{-1}<m$ and $q q^{-1} \equiv 1 ($mod $m)$), and $\gamma_0=-1$, $\gamma_1=0$, $\gamma_{i+1}=b_i \gamma_i - \gamma_{i-1}$. We have $\alpha_{i+1}\gamma_i - \alpha_i \gamma_{i+1}=-1$, $\beta_i = q \alpha_i - m \gamma_i$, and $\frac{m}{q^{-1}}=[b_s,\ldots,b_1]$. These numbers appear in the pull-back formulas $$ g'^*\big((u=0)\big) = \sum_{i=0}^{s+1} \beta_i E_i, \ \ \ \ \ \ \ g'^*\big((v=0)\big)= \sum_{i=0}^{s+1} \alpha_i E_i,$$ where $g':=\sigma \circ \eta \circ r$, and $K_{\widetilde{Y}}
\equiv \sigma^*(K_Y) + \sum_{i=1}^s (-1 +\frac{\beta_i+\alpha_i}{m}) E_i$. The numbers $d_i:= -1 + \frac{\beta_i + \alpha_i}{m}$ are the \textit{discrepancies} of $E_i$. Let $u_i$ be a local coordinate defining $E_i$, so that at each node of $E_0,E_1,\ldots,E_s,E_{s+1}$ we have local coordinates $u_i, u_{i+1}$ for $\widetilde{Y}$. 
Then (see \cite{R03}) we have that locally $g'$ is given by $$g'(u_i,u_{i+1})=(u_i^{\beta_i}u_{i+1}^{\beta_{i+1}},u_i^{\alpha_i}u_{i+1}^{\alpha_{i+1}}).$$ Therefore we have the pull-back relations $$du=\beta_i u_i^{\beta_i-1}u_{i+1}^{\beta_{i+1}} du_{i}+\beta_{i+1} u_i^{\beta_{i}}u_{i+1}^{\beta_{i+1}-1} du_{i+1}$$ and $$dv=\alpha_i u_i^{\alpha_i-1}u_{i+1}^{\alpha_{i+1}} du_{i}+\alpha_{i+1} u_i^{\alpha_{i}}u_{i+1}^{\alpha_{i+1}-1} du_{i+1}.$$

\subsection{Global picture} \label{s32} The following is taken from \cite[Section 1]{U10}. Let $X$ be a smooth projective surface over $\C$, and let $\sum_{j=1}^d D_j$ be a simple normal crossings divisor in $X$, that is, the irreducible curves $D_j$ are all smooth, and the singularities of the divisor are at most nodes. Let us assume the existence of a line bundle $\M$ on $X$ such that $$ \O_X \big( a_1 D_1 + a_2 D_2 + \ldots +a_d D_d \big) \simeq \M^{\otimes m}$$ for some integers $0<a_j<m$ such that gcd$(a_j,m)=1$. With this data, one constructs a smooth projective surface $Y'$ which represents the ``$m$-th root of $D:= \sum_{j=1}^d a_j D_j$" as follows. Let $s \in H^0(X,\O_X(D))$ be a section whose zero locus is $D$. This section defines a structure of $\O_X$-algebra on $\bigoplus_{j=0}^{m-1} \M^{-j}$ by means of the induced injection $\M^{-m} \simeq \O_X(-D) \hookrightarrow \O_X$. Then we have the affine morphism $f_0 \colon Y_0 \to X$, where $Y_0:= \spec_X \Big( \bigoplus_{i=0}^{m-1} \M^{-i} \Big)$. The variety $Y_0$ might not be normal. To normalize it, we define the line bundles $$\M^{(i)} := \M^{i} \otimes \O_X \Big(-\sum_{j=1}^d \Big[\frac{a_j i}{m} \Big] D_j \Big) $$ on $X$ for $0\leq i <m$. Then $\eta \colon Y:= \spec_X \Big( \bigoplus_{i=0}^{m-1} \M^{-(i)} \Big) \to Y_0$ is the normalization of $Y_0$. Hence if $f \colon Y \to X$ is the composition of $\eta$ with $f_0$, then $f_* \O_Y = \bigoplus_{i=0}^{m-1} \M^{-(i)}$. We note that $Y$ may have only cyclic quotient singularities over the nodes of $\sum_{j=1}^d D_j$. More precisely, given a node in $D_i \cap D_j$, we have one singularity in $Y$ over that node (since gcd$(a_j,m)=1$ for all $j$), and it is of type $\frac{1}{m}(1,q)$ where $a_i q+a_j \equiv 0$ modulo $m$. Locally around that singularity, the map $f \colon Y \to X$ is isomorphic to the local picture described in Subsection \ref{s31}, i.e. it is $\C^2/\langle \tau \rangle \to \C^2$ where $\{u=0\}=D_i$ and $\{v=0\}=D_j$.

Let $\sigma \colon Y' \to Y$ be the minimal resolution of the singularities in $Y$. The surface $Y'$ is a smooth (irreducible) projective surface. Let $f' \colon Y' \to X$ be the composition of $\sigma$ with $f$. Then $f'_* \O_{Y'} = \bigoplus_{i=0}^{m-1} \M^{-(i)}$. Again, the local picture of $f'$ over a node of $D_i \cap D_j$ is as in Subsection \ref{s31}, and so $f$ is locally isomorphic to $\sigma \circ \eta \circ r$.

Now let us consider $\omega \in H^0(X, \L \otimes S^r \Omega_X^1)$ for some line bundle $\L$ on $X$, and some integer $r>0$.       

\begin{theorem}
Assume that $D_j$ is $\omega$-integral (Definition \ref{omega}) for all $j$. If $\sum_{j=1}^d D_j-m \L$ is ample and $a_j \not \equiv -a_{j'}$ modulo $m$ for all $j \neq j'$, then $Y$ can have curves curves of geometric genus $\leq 1$ only in the set of preimages of $\omega$-integral curves in $X$.

\label{cyclic}
\end{theorem}

\begin{proof}
The proof follows the strategy of Section \ref{s1}: Lemma \ref{newomega}, Theorem \ref{gencub}, and Corollary \ref{norateli}. As in Lemma \ref{newomega}, let us prove the existence of $\omega''$ in $H^0(Y', \O_{Y'}(-(m-1)R-E) \otimes f'^* \L \otimes S^r \Omega_{Y'}^1)$, where $R$ is the sum of the strict transforms of the $D_j$, and $E=\sum_k E_k$ is the sum of all exceptional curves of $\sigma$. This is a local computation, and so let $D_i=\{u=0\}$ and $D_j=\{v=0\}$ at a node of $D_i \cap D_j$ in $X$. We assume that the cyclic quotient singularity is $\frac{1}{m}(1,q)$ with continued fraction of length $s$. Following the proof of Lemma \ref{newomega}, we only need to check that the pull-back of $du^{\otimes r-k} \otimes dv^{\otimes k}$ by $f'$ vanishes on $E_l$, where $0<k<r$ and $0 \leq l \leq s+1$. According to the local computation in Subsection \ref{s31}, this happens if and only if $\alpha_l >1$ or $\beta_l >1$. So assume that $\alpha_l \leq 1$ and $\beta_l \leq 1$. If $\alpha_l=0$, then $l=0$ and so $\beta_l=\beta_0=m >1$ a contradiction. The same for $\beta_l=0$. If $\alpha_l =1$, then $l=1$ and $\beta_l=\beta_1=q \geq 1$. Hence $q=1$, but this singularity is $\frac{1}{m}(1,1)$ and so the multiplicities $a_i, a_j$ of $D_i, D_j$ respectively must satisfy $a_i+a_j \equiv 0$ modulo $m$. But this is contrary to our assumptions. Same for $\beta_l=1$. Therefore for any $l$ we have $\alpha_l>1$ or $\beta_l>1$.

On the other hand, we have the numerical equivalence  
$$f'^*\Big(\sum_{j=1}^d D_j \Big) \equiv mR + m \sum_k (1+d_k)E_k,$$ where $d_k$ is the discrepancy associated to $E_k$ (see the end of Subsection \ref{s31}). Hence we obtain $$\frac{1}{m} f'^*\Big(-\sum_{j=1}^d D_j \Big) - (m-2)R + \sum_k d_k E_k + f'^* \L \equiv -(m-1)R-E +f'^* \L.$$ We recall that $-1<d_k <0$ for all $k$. 

Let $\N:= -(m-1)R-E +f'^* \L$, let $C \subset Y'$ be a curve of geometric genus $g$, and not exceptional for $\sigma$. Let $\varphi_C \colon \tilde{C} \to Y'$ be the normalization of $C \subset Y'$. Then $$\deg_{\tilde{C}} \Big(\varphi_C^* \N \otimes S^r\Omega_{\tilde{C}}^1 \Big) \leq \frac{1}{m}\Big(-\sum_{j=1}^d D_j + m\L \Big) \cdot f'(C) + r(2g-2)$$ by the projection formula. By our hypothesis we have $$\Big(-\sum_{j=1}^d D_j + m\L \Big) \cdot f'(C) <0,$$ and so if $g\leq 1$, then $\deg_{\tilde{C}} \Big(\varphi_C^* \N \otimes S^r\Omega_{\tilde{C}}^1 \Big)<0$, and so the curve $C$ is $\omega''$-integral. As in Lemma \ref{newomega}, we conclude that $f'(C)$ must be an $\omega$-integral curve in $X$. Since $\sigma \colon Y' \to Y$ is a birational morphism contracting $E$, we obtain that $Y$ can have curves of geometric genus $\leq 1$ only in the set of preimages of $\omega$-integral curves in $X$.

\end{proof}

\begin{remark}
The assumption $a_i+a_j \not \equiv 0$ modulo $m$ is to avoid the situation of cyclic quotient singularities of type $\frac{1}{m}(1,1)$. As we saw in the proof and in Sections \ref{s1} and \ref{s2} for the $A_1$ rational double points, the singularities $\frac{1}{m}(1,1)$ do not work for the existence of $\omega''$.
\end{remark}

We finish this section with an explicit example, where $X=\P^2$ and $D_j$ are lines. Let us consider the lines $ L_{t,u} = (t^2x + tuy + u^2 z) \subset \P^2$ for $[t,u] \in \P^1$. They are precisely the tangent lines to the conic $ (y^2-4xz) \subset \P^2$. Let $\{L_1,\ldots,L_{d} \}$ be distinct lines such that $L_i=L_{t_i,u_i}$ for some $[t_i,u_i] \in \P^1$. Let us take positive integers $a_1,\ldots,a_d$ such that $\sum_{i=1}^d a_i =mR$ for some integers $m,R>0$. Assume that $a_i<m$ and gcd$(a_i,m)=1$ for all $i$, and that $a_i+a_j$ is not divisible by $m$ for all $i \neq j$.

\begin{corollary}
If $4m < d$, then the surface $$(t_1^2x+t_1 u_1 y+u_1^2z)^{a_1}\cdots(t_d^2x+t_d u_d y+u_d^2z)^{a_d}=w^m$$ in $\P(1,1,1,R)$  contains no curves of geometric genus $\leq 1$ apart from the $\P^1$'s defined by $t_i^2x+t_i u_i y+u_i^2z=0$. Its normalization is a simply connected normal projective surface with ample canonical class.
\label{plane} 
\end{corollary}

\begin{proof}
First we need to indicate $\L$, $r$ and $\omega$. By taking the differential of $t_1^2x+t_1 u_1 y+u_1^2z$ for $u=1$ and $z=1$, we obtain the differential $$\omega = {dx}^2-y dx dy +x {dy}^2$$ which is a global section of $H^0(\P^2,\O_{\P^2}(4) \otimes S^2\Omega_{\P^2}^1)$, and so $\L:=\O(4)$ and $r=2$. Since we are considering this particular $\omega$, we know that the lines $(t_1^2x+t_1 u_1 y+u_1^2z)$  are all $\omega$-integral. By essentially \cite[Theorem 3.76]{GF15}, these lines are all the $\omega$-integral curves together with the discriminant curve $(y^2-4xz)$.

For the cyclic cover, we are considering $D_j:=L_j$ for all $j$, and $$\O_{\P^2}\Big(\sum_{j=1}^d a_j D_j \Big) \simeq \O_{\P^2}(R)^{\otimes m},$$ and so $\M:=\O_{\P^2}(R)$. Note that the variety $Y_0$ can be considered as $$Y_0 = \{(t_1^2x+t_1 u_1 y+u_1^2z)^{a_1}\cdots(t_d^2x+t_d u_d y+u_d^2z)^{a_d}=w^m\} \subset \P(1,1,1,R).$$ We also have that $\sum_{j=1}^d D_j-m\L=\O(d-4m)$, and so it is ample by assumption. Then we can apply Theorem \ref{cyclic}, and we obtain that $Y$ has curves of geometric genus $\leq 1$ only in the set of preimages of $\omega$-integral curves in $\P^2$. A simple calculation with Riemann-Hurwitz says that the only preimages of $\omega$-integral curves which give a curve of geometric genus $0$ or $1$ are the ramification curves, with genus $0$ indeed. With the normalization map $\eta \colon Y \to Y_0$, we have no modifications on geometric genus of curves, so the same statement holds for $Y_0$. The claim on simply connectedness and ampleness of canonical class for $Y$ follows from \cite[Theorem 8.5]{U10} and the Canonical class formula \cite[Proposition 1.4]{U10}, which is generalized Riemann-Hurwitz.  

\end{proof}

\section{Low genus curves in towers of cyclic covers} \label{s4}

In this section we put all together to give the construction of a wide range of algebraic surfaces in which we can control curves of geometric genus $\leq 1$.  

\begin{theorem}
Let $X$ be a smooth projective surface, and let $\omega \in H^0(X, \L \otimes S^r\Omega_X^1)$. Assume we have the relations $m_i \M_i = \sum_{j=1}^{s_i} a_{i,j} D_{i,j}$ in Pic$(X)$, for some line bundles $\M_i$, with $0<a_{i,j}<m_i$ and gcd$(a_{i,j},m_i)=1$ for all $i,j$. Assume also that the divisor $\sum_{i=1}^n \sum_{j=1}^{s_i} D_{i,j}$ has simple normal crossings, and $D_{i,j}$ are $\omega$-integral curves.  Then a tower of $n$ cyclic covers of degree $m$ $$X_n \to X_{n-1} \to \ldots \to X_1 \to X_0:= X$$ is defined, where all $X_k$ are normal projective surfaces with only cyclic quotient singularities. 

If $$\sum_{i=1}^n \frac{1} {m_i} \Big( \sum_{j=1}^{s_i} D_{i,j}\Big) -\L \ \ \text{is} \ \Q\text{-ample, and } \ a_{i,j} \not \equiv -a_{i,j'} (mod \ m_i) \ \text{ for all} \ i, \ j\neq j',$$ then $X_n$ can have curves of geometric genus $\leq 1$ only in the set of preimages of $\omega$-integral curves in $X$. 

\label{general}
\end{theorem}

\begin{proof}
The expressions $m_i \M_i = \sum_{j=1}^{s_i} a_{i,j} D_{i,j}$ in Pic$(X)$ define a tower of $n$ cyclic covers of degree $m_i$ 
$$X_n \to X_{n-1} \to \ldots \to X_1 \to X_0:= X$$ inductively as follows: Let $f_{k+1} \colon X_{k+1} \to X_{k}$ be the cyclic cover defined by $$ g_k^* \O_X \big(\sum_{j=1}^{s_{k+1}} a_{k+1,j} D_{k+1,j} \big) \simeq {g_k^*\M_i}^{\otimes m_i},$$ where $g_k:= f_1 \circ \cdots \circ f_k$ as done in Subsection \ref{s32}. We note that if $D'_{k+1,j}$ is the strict transform of $D_{k+1,j}$ under $g_k$, then $g_k^* \O_X \big(\sum_{j=1}^{s_{k+1}} a_{k+1,j} D_{k+1,j} \big)= \O_{X_k} \big( \sum_{j=1}^{s_{k+1}} a_{k+1,j} D'_{k+1,j} \big)$, and $\sum_{j=1}^{s_{k+1}} D'_{k+1,j}$ is a simple normal crossings divisor.
All $X_k$ are normal projective surfaces with cyclic quotient  singularities. Let $\sigma_k \colon X'_k \to X_k$ be the minimal resolution of all singularities in $X_k$. The surface $X'_k$ is a smooth (irreducible) projective surface. Let $g'_k \colon X'_k \to X$ be the composition of $\sigma_k$ with $g_k$.

The proof follows again the strategy of Section \ref{s1}: Lemma \ref{newomega}, Theorem \ref{gencub}, and Corollary \ref{norateli}. As in Theorem \ref{cyclic}, one proves the existence of $\omega''$ in $H^0(X'_n, \O_{X'_n}(-R-E) \otimes {g'}_n^* \L \otimes S^r \Omega_{X'_n}^1)$, where $R=\sum_{i=1}^n (m_i-1) R_i$ and $R_i$ is the sum of the strict transforms of the $\sum_{j=1}^{s_i} D_{i,j}$, and $E=\sum_k E_k$ is the sum of all exceptional curves of $\sigma_n$. For that it is key the hypothesis $a_{i,j} \not \equiv -a_{i,j'} ($mod $m_i)$ for all $i, j \neq j'$.

On the other hand, we have the numerical equivalence  
$${g'}_n^*\Big(\sum_{j=1}^{s_i} D_{i,j} \Big) \equiv m_i R_i  + m_i  \sum_{k_i} (1+d_{k_i})E_{k_i},$$ where the sum of exceptional curves runs over the singularities due to $\sum_{j=1}^{s_i} D_{i,j}$, and $d_{k_{i}}$ is the discrepancy associated to $E_{k_{i}}$ (see the end of Subsection \ref{s31}). 

Hence we obtain that $$ -{g'}_n^* \Big( \sum_{i=1}^n \frac{1} {m_i} \Big( \sum_{j=1}^{s_i} D_{i,j}\Big) \Big) - \sum_{i=1}^n (m_i-2) R_i + \sum_k d_k E_k + {g'}_n^*\L$$ is numerically equivalent to $\N:=-R-E +{g'}_n^* \L$. We recall that $-1<d_k <0$ for all $k$. 

Let $C \subset X'_n$ be a curve of geometric genus $g$, and not exceptional for $\sigma_n$. Let $\varphi_C \colon \tilde{C} \to X'_n$ be the normalization of $C \subset X'_n$. Then $$\deg_{\tilde{C}} \Big(\varphi_C^* \N \otimes S^r\Omega_{\tilde{C}}^1 \Big) \leq \Big( -\sum_{i=1}^n \frac{1} {m_i} \Big( \sum_{j=1}^{s_i} D_{i,j}\Big) + \L\Big) \cdot g'_n(C) + r(2g-2)$$ by the projection formula. By our hypothesis we have $$ \Big( - \sum_{i=1}^n \frac{1} {m_i} \Big( \sum_{j=1}^{s_i} D_{i,j}\Big) +\L \Big) \cdot g'_n(C) <0,$$ and so if $g\leq 1$, then $\deg_{\tilde{C}} \Big(\varphi_C^* \N \otimes S^r\Omega_{\tilde{C}}^1 \Big)<0$, and so the curve $C$ is $\omega''$-integral. As in Lemma \ref{newomega}, we conclude that $g'_n(C)$ must be an $\omega$-integral curve in $X$. Since $\sigma_n \colon X'_n \to X_n$ is a birational morphism contracting $E$, we obtain that $X_n$ can have curves of geometric genus $\leq 1$  only in the set of preimages of $\omega$-integral curves in $X$.
\end{proof}

\section{hyperbolicity} \label{s5}

By using certain result in Nevanlinna theory for complex varieties (cf.\ \cite{V00,V00b}), we will prove that if 
the normal projective surface $X_n$ in Theorem \ref{general} contains no curves of geometric genus $\leq 1$ and $\omega=0$ has only algebraic solutions, then in fact $X_n$ is hyperbolic, that is, the only holomorphic maps $f\colon \C \to X_n$ are constant. These in practice produce many families of smooth projective surfaces of general type which are hyperbolic. For example, at the end of this section we prove existence of hyperbolic complete intersection surfaces of low degrees.

We recall that an \emph{entire curve} in a variety $X$ is a nonconstant holomorphic map $\C \to X$. We begin with some definitions in Nevanlinna theory (cf.\ \cite[Section 11]{V11}). Let $X$ be a smooth projective surface, let $f\colon \C\to X$ be an entire curve, and let $D$ be a divisor on $X$ whose support does not contain the image of $f$. We define the \emph{counting function} of $D$ in $X$ to be $$N_f(D,R)=\sum_{0<|z|<R}\mathrm{ord}_zf^*D\cdot \mathrm{log}\Big( \frac{R}{|z|} \Big)+\mathrm{ord}_0f^*D\cdot\mathrm{log}(R).$$
Let $\lambda$ be a Weil function for $D$. We define the \emph{proximity function} for $f$ relative to $D$ to be $$m_f(D,R)=\int_0^{2\pi}\lambda(f(Re^{i\theta}))\frac{d\theta}{2\pi}.$$ It is defined up to $O(1)$.
We can now define the \emph{height} of $f$ relative to $D$ 
by
$$T_{D,f}(R)=m_f(D,R)+N_f(D,R).$$
The \emph{height} of $f$ relative to an invertible sheaf $\mathcal{L}$ on $X$ is $T_{\mathcal{L},f}(R)=T_{D,f}(R)+O(1)$, where $D$ is any divisor such that $\mathcal{L} \simeq \mathcal{O}_X(D)$.

The following result is \cite[Corollary 5.2]{V00b} for $D=0$ (see also \cite[Proposition 6.1.1]{V00}). We use the notation $\mathrm{log}^+(a)=\mathrm{max}\{0,\mathrm{log}(a)\}$.

\begin{theorem}\label{vhyp}
Let $X$ be a smooth complex projective variety, let $f \colon \C\to X$ be an entire curve, let $r$ be a positive integer, let $\mathcal{L}$ be a line bundle on $X$, let $\omega$ be a global section of $\mathcal{L}^\vee\otimes S^r\Omega^1_{X/\mathbb{C}}$, and let $\mathcal{A}$ be a line bundle which is big on the Zariski closure of $f(\C)$. If $f^* \omega \neq 0$, then $$T_{\mathcal{L},f}(R)\leq_{exc} O(\mathrm{log}^+T_{\mathcal{A},f}(R))+o\big(\mathrm{log} (R) \big),$$ where the notation $\leq_{exc}$ means that the inequality holds for all $R>0$ outside of a set of finite Lebesgue measure.
\label{vojtakey}
\end{theorem}

The following will be the main tool to prove hyperbolicity for surfaces.

\begin{corollary}
Let $X$ be a smooth projective surface, let $\mathcal{L}$ be a big line bundle on $X$, let $r>0$ be an integer, and let $\omega \in H^0(X,\mathcal{L}^{\vee} \otimes S^r\Omega^1_{X})$. Assume that $f \colon \C \to X$ is an entire curve whose image is Zariski dense. Then $f^* \omega=0$.
\label{hyperbolickey}
\end{corollary}

\begin{proof}
By \cite[Prop. 11.11]{V11}, we have that $T_{\O(1),f}(R) \leq C T_{\L,f}(R) + O(1)$ for all $R >0$, and a constant $C>0$ depending on $\O(1)$ and $\L$. Here $\O(1)$ is the hyperplane line bundle given by some embedding of $X$ in a projective space. On the other hand, we have that there are constants $M>0$ and $N$ such that $M \log(R) + N \leq T_{\O(1),f}(R)$ for all $R>0$, and so there are constants $M'>0$ and $N'$ such that $M' \log(R) +N \leq T_{\L,f}(R)$. 

Let $f^* \omega \neq 0$. From Theorem \ref{vojtakey} we have that $$T_{\L,f}(R) \leq_{exc} S \log(T_{\L,f}(R)) + \epsilon \log(R)$$ for a constant $S>0$, a given $0<\epsilon< M'/4$, and $R >>0$ out of a set of finite Lebesgue measure. But $ \log(T_{\L,f}(R)) < T_{\L,f}(R)/2S$ for $r>>0$, and so by the inequality above we get $T_{\L,f}(R) <_{exc} 2\epsilon \log(R) < M'/2 \log (R)$, for certain $R>>0$. But this contradicts $M' \log(R) +N \leq T_{\L,f}(R)$. Therefore $f^* \omega =0$.  
\end{proof}

The following theorem gives a criteria for hyperbolicity of the singular surfaces $X_n$ in Theorem \ref{general}.

\begin{theorem}
Let us consider the hypothesis and the notation in Theorem \ref{general}. In addition, assume that all solutions to the differential equation given by $\omega=0$ on $X$ are $\omega$-integral curves. Then an entire curve in $X_n$ must be contained in the set of preimages of $\omega$-integral curves in $X$.

In particular, if the set of preimages of $\omega$-integral curves in $X$ does not contain curves of geometric genus $0$ or $1$, then $X_n$ is hyperbolic.
\label{hyperbolic}
\end{theorem}

\begin{proof}
Let $\sigma_n \colon X'_n \to X_n$ be the minimal resolution of singularities of $X_n$. Consider an entire curve $f \colon \C \to X_n$. Then it has a lifting $f' \colon \C \to X'_n$. Assume that $f'(\C)$ is Zariski dense in $X'_n$. We have a section $$\omega'' \in H^0(X'_n, \O_{X'_n}(-R-E) \otimes {g'}_n^* \L \otimes S^r\Omega_{X'_n}^1),$$ and a line bundle $\N:= \O_{X'_n}(-R-E) \otimes {g'}_n^* \L$ which is numerically $$\N^{\vee} \equiv {g'}_n^* \Big(-\L+  \sum_{i=1}^n \frac{1} {m_i} \Big( \sum_{j=1}^{s_i} D_{i,j}\Big) \Big) + \sum_{i=1}^n (m_i-2) R_i - \sum_k d_k E_k,$$ where $-1 < d_k <0$ are the discrepancies of $E_k$, and $-\L+  \sum_{i=1}^n \frac{1} {m_i} \Big( \sum_{j=1}^{s_i} D_{i,j}\Big)$ is an ample divisor in $X$ by hypothesis. Therefore $\N^{\vee}$ is numerically the sum of the pull-back of an ample divisor plus an effective divisor. It is easy to see that $\N^{\vee}$ is then big by e.g.\ \cite[Corollary 2.2.7]{Laz}. Therefore by Corollary \ref{hyperbolickey} we have that $f^* \omega''=0$. But then, locally analytical $f$ satisfies the differential equation given by ${g'}_n^* \omega$, and then this gives a solution to the differential equation given by $\omega$. By hypothesis we know that all solutions are given by $\omega$-integral curves in $X$, and so this contradicts the Zariski density of $f(\C)$.

Therefore $f(\C)$ must be contained in an irreducible algebraic curve. But by Liouville's theorem, the geometric genus of this irreducible curve must be less than or equal to $1$. Moreover, by Theorem \ref{general}, all curves of geometric genus $\leq 1$ are in the set of preimages of $\omega$-integral curves in $X$.   
\end{proof}

The following corollaries give a proof of Theorem \ref{jdjhds}.  

\begin{corollary}
Let $n \geq 3$, and let $m_i \geq 3$ be $n$ integers. Let $\{a_{i,j}\}$ and $\{b_{i,j}\}$ be two collections of distinct $\sum_{i=1}^n m_i$ complex numbers. Let $\{ G_i=G_i(z_0,z_1,z_2,z_3) \}_{i=1}^n$ be a collection of $n$ homogeneous polynomials of degree $m_i$, such that $z_0-a_{i,j}z_1-b_{i,j}z_2+ a_{i,j} b_{i,j} z_3$ does not divide $G_i$. Then the complete intersection $$ \prod_{j=1}^{m_i}(z_0-a_{i,j} z_1-b_{i,j} z_2+a_{i,j} b_{i,j} z_3) + t_i G_i =z_{3+i}^{m_i},  \ \ z_0z_3-z_1z_2=0 $$ for $i=1,\ldots,n$ in $\P^{n+3}$ is hyperbolic for sufficiently small $t_i \in \C$.
\label{CompleteIntersectionsD}
\end{corollary}

\begin{proof}
Let us evaluate the complete intersection $$ \prod_{j=1}^{m_i}(z_0-a_{i,j} z_1-b_{i,j} z_2+a_{i,j} b_{i,j} z_3) + t_i G_i =z_{3+i}^{m_i},  \ \ z_0z_3-z_1z_2=0 $$ for $i=1,\ldots,n$ in $\P^{n+3}$ in $t_i=0$ for all $i$. We denote this surface by $X_n$, which is as in the construction of the generalized surfaces of cuboids but for not necessarily equal degrees $m_i>2$ (see Section \ref{s1}). As in Corollary \ref{norateli}, this surface $X_n$ has no curves of geometric genus $\leq 1$. In fact, the construction satisfies the hypothesis in Theorem \ref{hyperbolic}, and so $X_n$ has no entire curves. We now consider an small deformation of the branched divisor. It gives a smooth branch divisor, and so a smooth surface $X_n(t_1,\ldots,t_n)$. At the same time, this gives a small perturbation of the hyperbolic surface $X_n$, and it is known that hyperbolicity is preserved by small deformations (see \cite[p.148 Theorem (3.11.1)]{Kob}). This is a proof of part (d) of Theorem \ref{jdjhds}. 
\end{proof}

Next corollary is proved as the previous one, but we need to take care of the degrees equal to $2$, which produce $A_1$ singularities. 

\begin{corollary}
Let $r \geq 1$, $s\geq 1$ be integers such that $r+s \geq 5$, and let $\{m_i \geq 3\}_{i=r+1}^{r+s}$ be integers. Let $\{b_{i,j}, c_{i,j} \}_{i=r+1,\ldots,r+s}^{j=1,\ldots,m_i}$ be two collections of $\sum_{i=r+1}^{r+s} m_i$ distinct complex numbers. Let $a_i$ be a collection of $r$ distinct complex numbers such that $a_i \neq \pm a_j$, $b_{i,j} \neq \pm a_k$, $c_{i,j} \neq \pm a_k$, for all $i,j,k$. Let $\{ F_i=F_i(z_0,z_1,z_2,z_3) \}_{i=1}^{r}$ be a collection of $r$ homogeneous polynomials of degree $2$, such that $z_0-a_i z_1-a_i z_2+ a_{i}^2 z_3$ and $z_0+a_i z_1+a_i z_2+ a_{i}^2 z_3$ do not divide $F_i$. Let $\{ G_i=G_i(z_0,z_1,z_2,z_3) \}_{i=r+1}^{r+s}$ be a collection of $s$ homogeneous polynomials of degree $m_i$, such that $z_0-b_{i,j}z_1-c_{i,j}z_2+ b_{i,j} c_{i,j} z_3$ does not divide $G_i$. 

Then the complete intersection defined by $$(z_0-a_i z_1-a_i z_2+ a_{i}^2 z_3)(z_0+a_i z_1+a_i z_2+ a_{i}^2 z_3) + t_i F_i = z_{3+i}^2 $$ $$ \prod_{j=1}^{m_i}(z_0-b_{i,j} z_1-c_{i,j} z_2+b_{i,j} c_{i,j} z_3) + t_i G_i =z_{3+i}^{m_i},  \ \ z_0z_3-z_1z_2=0 $$ for $i=1,\ldots,r+s$ in $\P^{r+s+3}$ is hyperbolic for sufficiently small $t_i \in \C$.

\label{CompleteIntersectionsC}
\end{corollary}

\begin{proof}
Let us evaluate the complete intersection at $t_i=0$, and denote this surface by $X_{r+s}$. Then $X_{r+s}$ is constructed from $X_0:= \P^1 \times \P^1$ as for the generalized cuboids but for distinct degrees $2$ and $m_i$. As before, we consider the section $\omega \in H^0(X_0, (2,2) \otimes S^2 \Omega^1_{X_0})$ defined by $z_3^2 dz_1 dz_2$. We recall that given the isomorphism $h \colon \P^1 \times \P^1 \to X_0,$ $h([x,y] \times [w,z])=[xw,xz,yw,yz]$, the section $\omega$ corresponds to the section $y^2z^2 dx dw$ under $h$. The problem is that over multiplicities equal to $2$, we obtain $A_1$ singularities and we cannot apply our results, like the key Lemma \ref{newomega}. Instead we consider another section so that we can use the result in Lemma \ref{newomega2}. For that, let $$\omega_0 \in H^0(X_0, (4,4) \otimes S^2 \Omega_{X_0}^1)$$ be defined by $(x^2-w^2)dx dw$ for affine coordinates $x,w$. Then if $R$ is the strict transform of the branch divisor in $X_{r+s}'$, and $E$ is the exceptional divisor of $X_{r+s}' \to X_{r+s}$, then there is $$\omega_0' \in H^0(X_{r+s}',\O_{X_{r+s}'}(-R-E) \otimes {g'}_{r+s}^*(4,4) \otimes S^2 \Omega_{X_{r+s}'}^1)$$ corresponding to ${g'}_{r+s}^{\bullet} \omega_0$. At the same time we have $${g'}_{r+s}^*(r+s,r+s) \equiv R+E,$$ and so $\O_{X_{r+s}'}(-R-E) \otimes {g'}_{r+s}^*(4,4) \equiv {g'}_{r+s}^*(-r-s+4,-r-s+4)$. But $r+s \geq 5$, and so we obtain as in Theorem \ref{general} and Theorem \ref{hyperbolic} that the surfaces $X_{r+s}$ are hyperbolic. This is indeed because we know all $\omega_0$-integral curves (fibres and the two $(1,1)$ extra curves), and so we can check all the pre-images. To avoid curves of geometric genus $\leq 1$, here we use that $s>0$, $m_i \geq 3$, and $s+r\geq 5$. For example $s=0$ would be a problem with the $(1,1)$ curves. This is a proof of part (c) of Theorem \ref{jdjhds}. 
\end{proof}

The next corollary uses the example at the end of Section \ref{s3}.

\begin{corollary}
Let $n \geq 5$, and let $m_i \geq 3$ be $n$ integers. Let $[a_i,b_i]$ be a collection of distinct $\sum_{i=1}^n m_i$ points in $\P^1$. Let $\{ G_i=G_i(x,y,z) \}_{i=1}^n$ be a collection of $n$ homogeneous polynomials of degree $m_i$, such that $a_j^2x+a_j b_j y + b_j^2z$ does not divide $G_i$ for $j=m_{i-1},m_{i-1}+1,\ldots, m_i$ (where $m_0:=1$). Then the complete intersection $$ \prod_{j=1}^{m_i}(a_j^2x+a_j b_j y + b_j^2z) + t_i G_i =w_i^{m_i}$$ for $i=1,\ldots,n$ in $\P^{n+2}$ is hyperbolic for sufficiently small $t_i \in \C$.
\label{CompleteIntersectionsB}
\end{corollary}

\begin{proof}
Let us consider instead the situation in Corollary \ref{plane}, but with $n\geq 5$ equations, $a_i=1$ for all $i$, and $m_i\geq 3$ for all $i=1,\ldots,n$. We take $X_0:=\P^2$, $\omega \in H^0(\O_{\P^2}(4) \otimes S^2 \Omega_{\P^2}^1)$, and we construct the surface $X_n$ as in Theorem \ref{general} from this data. As in Corollary \ref{CompleteIntersectionsC}, for $n\geq 5$ and $m_i \geq 3$, we construct smooth complete intersections in $\P^{n+2}$ of multidegree $(m_1,\ldots,m_n)$ which are hyperbolic, by Theorem \ref{hyperbolic}. The proof follows the same strategy as Corollary \ref{CompleteIntersectionsC}, since we know all $\omega$-integral curves. This is a proof of part (b) of Theorem \ref{jdjhds}, and part (a) when all multiplicities are bigger than or equal to $3$.  
\end{proof}

\begin{corollary}
Let $r \geq 1$, $s\geq 0$ be integers such that $r+s \geq 7$, and let $\{m_i \geq 3\}_{i=r+1}^{r+s}$ be integers. Let $\{b_{i,j}, c_{i,j} \}_{i=r+1,\ldots,r+s}^{j=1,\ldots,m_i}$ be two collections of $\sum_{i=r+1}^{r+s} m_i$ distinct complex numbers. Let $a_i$ be a collection of $r$ distinct complex numbers such that $a_i \neq a_j \pm 1$, $b_{i,j} \neq a_k$, $c_{i,j} \neq a_k$, $ b_{i,j} \neq a_k \pm 1$, and $c_{i,j} \neq a_k \pm 1$ for all $i,j,k$. Let $\{ F_i=F_i(z_0,z_1,z_2,z_3) \}_{i=1}^{r}$ be a collection of $r$ homogeneous polynomials of degree $2$, such that $z_0-a_i z_1-a_i z_2+ a_{i}^2 z_3$ and $z_0-(a_i-1) z_1- (a_i+1) z_2+ (a_{i}^2-1) z_3$ do not divide $F_i$. Let $\{ G_i=G_i(z_0,z_1,z_2,z_3) \}_{i=r+1}^{r+s}$ be a collection of $s$ homogeneous polynomials of degree $m_i$, such that $z_0-b_{i,j}z_1-c_{i,j}z_2+ b_{i,j} c_{i,j} z_3$ does not divide $G_i$. 

Then the complete intersection defined by $$(z_0-a_i z_1-a_i z_2+ a_{i}^2 z_3)(z_0-(a_i-1) z_1- (a_i+1) z_2+ (a_{i}^2-1) z_3) + t_i F_i = z_{3+i}^2 $$ $$ \prod_{j=1}^{m_i}(z_0-b_{i,j} z_1-c_{i,j} z_2+b_{i,j} c_{i,j} z_3) + t_i G_i =z_{3+i}^{m_i},  \ \ z_0z_3-z_1z_2=0 $$ for $i=1,\ldots,r+s$ in $\P^{r+s+3}$ is hyperbolic for sufficiently small $t_i \in \C$.
\label{CompleteIntersectionsA}
\end{corollary}

\begin{proof}
This is as in the proof of Corollary  \ref{CompleteIntersectionsC}, but with $\omega_0 \in H^0(X_0, (6,6) \otimes S^2 \Omega_{X_0}^1)$ defined by $$(x-w)(x-w+1)(x-w-1)(x-w-2)dx dw$$ for affine coordinates $x,w$. We note that each of the $4$ nodes in the $4$ fibres $(z_0-a_i z_1-a_i z_2+ a_{i}^2 z_3)(z_0-(a_i-1) z_1- (a_i+1) z_2+ (a_{i}^2-1) z_3)=0$ in $X_0=\P^1 \times \P^1$ belongs to one of the $4$ lines $(x-w)(x-w+1)(x-w-1)(x-w-2)=0$. This gives that the pull-back of each of the $4$ lines do not contain any curves of geometric genus $\leq 1$. This is a proof of part (a) of Theorem \ref{jdjhds}, when some (or all) multiplicities are equal to $2$. 
\end{proof}

The list of multidegrees not included in the previous corollaries is $(2,\ldots,2)$ in $\P^9$, $\P^8$, and $\P^7$; $(m_1,m_2,m_3,m_4)$ for $m_i>2$ and $(2,2,2,2)$ in $\P^6$, for which the surface of cuboids is an example. We do not know existence of hyperbolic surfaces in those cases, except for the ones $(k,k,k,k) \subset \P^6$ for $k\geq 3$ \cite{GF16}. This gives explicit evidence for \cite[Conjecture 0.18]{Dem18} in the case of surfaces. We point out that hyperbolic complete intersections of high multidegree have been constructed by Brotbek \cite{Br14} (see also \cite{X15}). 


\end{document}